\documentclass[11pt,twoside]{article}
\usepackage{amsfonts,graphicx,lscape}
\usepackage{wrapfig}
\usepackage{color}
\usepackage{graphics,subfigure}
\usepackage{epsf}
\usepackage{epsfig}

\usepackage{amssymb}
\usepackage{eucal}
\usepackage{amsmath}
\usepackage{amsthm}
\usepackage{amscd}
\usepackage{latexsym}
\usepackage{booktabs}
\usepackage{epstopdf}
\usepackage{float}
\usepackage[all]{xy}

\newtheorem{thm}{Theorem}[section]

\newtheorem{proposition}[thm]{Proposition}
\newtheorem{remark}[thm]{Remark}
\title{\bf ON THE DYNAMICS OF A HAMILTON-POISSON SYSTEM\footnote{PREPRINT - Sci. Bull. Politehnica University of Timisoara, Transactions on Mathematics \& Physics, \textbf{63}(77), Issue 2, (2018), 14-–28.}}

\author{Cristian L\u azureanu and Camelia Petri\c sor\\
Politehnica University of Timi\c{s}oara}

\date{}

\begin{document}

\thispagestyle{plain}

\maketitle

\begin{abstract}
The dynamics of a three-dimensional Hamilton-Poisson system is closely related to its constants of motion, the energy or Hamiltonian function $H$ and a Casimir $C$ of the corresponding Lie algebra. The orbits of the system are included in the intersection of the level sets $H=constant$ and $C=constant$. Furthermore, for some three-dimensional Hamilton-Poisson systems, connections between the associated energy-Casimir mapping $(H,C)$ and some of their dynamic properties were reported. In order to detect new connections, we construct a Hamilton-Poisson system using two smooth functions as its constants of motion. The new system has infinitely many Hamilton-Poisson realizations. We study the stability of the equilibrium points and the existence of periodic orbits. Using numerical integration we point out four pairs of heteroclinic orbits. \footnote{MSC(2010): 70H12, 70H14, 70K20, 70K42, 70K44, 65D30.

Keywords and phrases: {\it Hamilton-Poisson dynamics, energy-Casimir mapping, stability, periodic orbits, heteroclinic orbits, mid-point rule.}}
\end{abstract}
\bigskip

\section{Introduction}
\smallskip
Many systems of first order differential equations that model processes in physics, chemistry, biology, economy, and other domains are three-dimensional systems. Some of them have two constants of motion. Consequently, they are Hamilton-Poisson systems (see, for example \cite{Tud12}). The dynamics of such systems takes place at the intersection of the common level sets of the Hamiltonian and the Casimir (see, for example, \cite{HolmMar91}). In \cite{TudAroNic09}, the energy-Casimir mapping $\mathcal{EC}=(H,C)$ is introduced. Moreover, some connections between the dynamics of the system considered in \cite{TudAroNic09} and the associated energy-Casimir mapping were given. In recent papers, the same connections and new ones were reported (see, for example, \cite{Lazureanu17c}, \cite{LazPet18} and therein references). These connections depend on the image of the energy-Casimir mapping. Also they depend on the partition of this image given by the images of the equilibrium points through $\mathcal{EC}$. In the most of cases the image of $\mathcal{EC}$ is a convex set. In this paper we add a new example in the list of the systems that are analyzed by this point of view. We mention here that in our case the image of the energy-Casimir mapping is a non-convex set.

The paper is organized as follows. In Section 2, we construct a three-dimensional system of differential equations using two smooth functions as its constants of motion. One of these constants of motion is a Casimir of the Lie algebra $\mathfrak{so}(3)$, but the other one is a non-quadratic polynomial. We recall that quadratic Hamilton-Poisson systems on the dual space of $\mathfrak{so}(3)$ were investigated in \cite{Adams14,Adams16}. In Section 3, using the above-mentioned Lie algebra, we give a Hamilton-Poisson realization of the considered system. Moreover, we obtain that our system has infinitely many Hamilton-Poisson realizations. In Section 4, we consider the energy-Casimir mapping $\mathcal{EC}$ associated to the considered system. Using the images of critical points of $\mathcal{EC}$ we give a semialgebraic partition of the image of this mapping. The connections of the energy-Casimir mapping with the dynamics of our system are pointed out in next sections. In Section 5, we prove results regarding the stability of the equilibrium points. In Section 6, we establish the topology of the fibers of the energy-Casimir mapping. We prove the existence of the periodic orbits. Using numerical simulations, we also claim the existence of heteroclinic orbits.

\section{A construction of an integrable three-dimensional\\ system}
In this section we construct a three-dimensional system of differential equations using two smooth functions as its constants of motion. 

Let $H$ and $C$ be two smooth functions given by
\begin{equation}\label{H-C}
H(x,y,z)=\dfrac{1}{4}x^4+\dfrac{1}{4}y^4-\dfrac{1}{2}z^2~,~~C(x,y,z)=\dfrac{1}{2}x^2+\dfrac{1}{2}y^2+\dfrac{1}{2}z^2.
\end{equation}
Consider $x,y,z\in C^1(\mathbb{R})$ such that
\begin{align*}
&H(x(t),y(t),z(t))=H(x(0),y(0),z(0))\\
&C(x(t),y(t),z(t))=C(x(0),y(0),z(0))~,~\forall t\in\mathbb{R}.
\end{align*} 
Then $$\frac{dH}{dt}=0~,~~\frac{dC}{dt}=0,$$
that is 
\begin{align*}
&\frac{\partial H}{\partial x}\dot x+\frac{\partial H}{\partial y}\dot y=-\frac{\partial H}{\partial z}\dot z\\
&\frac{\partial C}{\partial x}\dot x+\frac{\partial C}{\partial y}\dot y=-\frac{\partial C}{\partial z}\dot z.
\end{align*}
In our case we have
\begin{align*}
&x^3\dot x+y^3\dot y=z\dot z\\
&x\dot x+y\dot y=-z\dot z.
\end{align*}
Setting $$\dot z=x^3y-xy^3,$$
we get the following system
\begin{equation}\label{1}
\left\{\begin{array}{l}
\dot x=yz(1+y^2)\\
\dot y=-xz(1+x^2)\\
\dot z=xy(x^2-y^2)
\end{array}\right.
\end{equation}
It is obvious that $H$ and $C$ are constants of motion of system \eqref{1}. Moreover, this system is integrable and, in fact, it is a Hamilton-Poisson system. 
\section{Hamilton-Poisson realizations}
In this section we give Hamilton-Poisson realizations of system \eqref{1}. We obtain that the considered system is bi-Hamiltonian and in addition it has infinitely many Hamilton-Poisson realizations.
\begin{proposition}\label{p1}
The system (\ref{1}) has the Hamilton-Poisson realization $$(\mathfrak{so}(3)^*,\Pi_1,H),$$ where $\mathfrak{so}(3)^*$ is the dual space of the Lie algebra $\mathfrak{so}(3)$, the Hamiltonian function $H$ is given by \eqref{H-C} and the Poisson structure is given by
\begin{equation}\label{pi1}
\Pi_1=\left[\begin{array}{rrr}
0&z&-y\\
-z&0&x\\
y&-x&0\end{array}\right].
\end{equation}
\end{proposition}
\begin{proof}
It is known that the function $C$ given by \eqref{H-C} is a Casimir of the Lie algebra $\mathfrak{so}(3)$ (see, e.g. \cite{Adams14}), where
$$
\mathfrak{so}(3)=\{X=\left[\begin{array}{ccc}
0 &-w&v \\
w &0&-u\\
-v &u &0
\end{array}\right]:u,v,w \in\mathbb{R}\}.$$
We immediately obtain $\Pi_1\cdot\nabla C={\bf 0}$ and $\Pi_1\cdot\nabla H=\dot{\bf x}^t$, ${\bf x}=(x,y,z)$.
\end{proof}
\begin{proposition}\label{p2}
The system (\ref{1}) is a bi-Hamiltonian system.
\end{proposition}
\begin{proof}
Considering the second Poisson structure
\begin{equation}\label{pi2}
\Pi_2=\left[\begin{array}{ccc}
0&z&y^3\\
-z&0&-x^3\\
-y^3&x^3&0\end{array}\right],
\end{equation}
it follows that system \eqref{1} has the Hamilton-Poisson realization $(\mathbb{R}^3,\Pi_2,C)$, where $C$ is given by \eqref{H-C}. Furthermore the function $H$ \eqref{H-C} fulfills $\Pi_2\cdot\nabla H={\bf 0}$.

Because $\Pi_1\cdot\nabla H=\Pi_2\cdot\nabla C=\dot{\bf x}^t$ and $\Pi_1$ and $\Pi_2$ are compatible Poisson structures, the conclusion follows.
\end{proof}
Using the above results, we obtain the Poisson structure $\Pi_{a,b}=a\Pi_1-b\Pi_2$, $a,b\in\mathbb{R}$. Consider $c,d\in\mathbb{R}$ such that $ad-bc=1$ and $H_{c,d}=cC+dH$, $C_{a,b}=aC+bH$. We have $\Pi_{a,b}\cdot\nabla H_{c,d}=\dot{\bf x}^t$ and $\Pi_{a,b}\cdot\nabla C_{a,b}={\bf 0}$. Therefore we have proven the next result.
\begin{proposition}\label{p3}
There exist infinitely many Hamilton-Poisson realizations of system \eqref{1}, namely $(\mathbb{R}^3,\Pi_{a,b},H_{c,d})$, where
\begin{equation*}\label{pi}
\Pi_{a,b}=\left[\begin{array}{ccc}
0&(a-b)z&-ay+by^3\\
(b-a)z&0&ax-bx^3\\
ay-by^3&-ax+bx^3&0\end{array}\right],
\end{equation*}
and
$$H_{c,d}(x,y,z)=\frac{d}{4}(x^4+y^4)+\frac{c}{2}(x^2+y^2)+\frac{c-d}{2}z^2,$$
for every $a,b,c,d\in\mathbb{R}$ such that $ad-bc=1$.
\end{proposition}

\section{Energy-Casimir mapping}
In the geometric frame given by Proposition \ref{p1}, in this section we study some properties of the energy-Casimir mapping $\mathcal{EC}$ associated to system \eqref{1}. We present the image of this mapping. In addition, using the critical points of $\mathcal{EC}$ we obtain a partition of the image of the energy-Casimir mapping. This partition gives some connections with the dynamics of the considered system. 

Consider the Hamiltonian $H$ and a Casimir function $C$ given by \eqref{H-C}. The energy-Casimir mapping is given below
\begin{equation}\label{EC}
\mathcal{EC}:\mathbb{R}^3\to\mathbb{R}^2~,~~
\displaystyle\mathcal{EC}(x,y,z)=\left(\dfrac{1}{4}x^4+\dfrac{1}{4}y^4-\dfrac{1}{2}z^2,\dfrac{1}{2}x^2+\dfrac{1}{2}y^2+\dfrac{1}{2}z^2\right).
\end{equation}
The image of the energy-Casimir mapping is the set
\begin{equation*}
\mbox{Im}(\mathcal{EC})=\left\{(h,c)\in\mathbb{R}^2|(\exists)(x,y,z)\in\mathbb{R}^3:\mathcal{EC}(x,y,z)=(h,c)\right\}.
\end{equation*}
\begin{proposition}\label{p4}
Let $\mathcal{EC}$ be the energy-Casimir mapping \eqref{EC} associated to system \eqref{1}. Then 
\begin{equation}\label{ImEC}
\mbox{Im}(\mathcal{EC})=\{(h,c)\in\mathbb{R}^2|c\geq -h,c\geq \sqrt{h}\}.
\end{equation}
\end{proposition}
\begin{proof}
The pair $(h,c)$ belongs to $\mbox{Im}(\mathcal{EC})$ if and only if the system 
$$\dfrac{1}{4}x^4+\dfrac{1}{4}y^4-\dfrac{1}{2}z^2=h~,~~\dfrac{1}{2}x^2+\dfrac{1}{2}y^2+\dfrac{1}{2}z^2=c$$
is compatible. Performing algebraic computations, we get the conclusion.
\end{proof}
\begin{remark}\label{r1}
The energy-Casimir mappings of some particular Hamilton-Poisson systems were studied in many papers. In some cases the image of $\mathcal{EC}$ is $\mathbb{R}^2$ \cite{BinLaz13a,TudGir12}, in other cases it is a closed convex subset of $\mathbb{R}^2$ \cite{BinLaz13b,Lazureanu17a,Lazureanu17b,Lazureanu17c,Lazureanu17d,LazBin12a,LazBin12b,LazBin17,TudAroNic09}. In \cite{LazPet18} $\mbox{Im}(\mathcal{EC})$ is not a closed set. In our case the image of the considered energy-Casimir mapping is shown in Figure \ref{fig:1}. It is a closed non-convex set, what explains while we choose those constants of motion given by \eqref{H-C}.
\end{remark}
\begin{figure}[h!]
\centering
\includegraphics[width=0.5\linewidth]{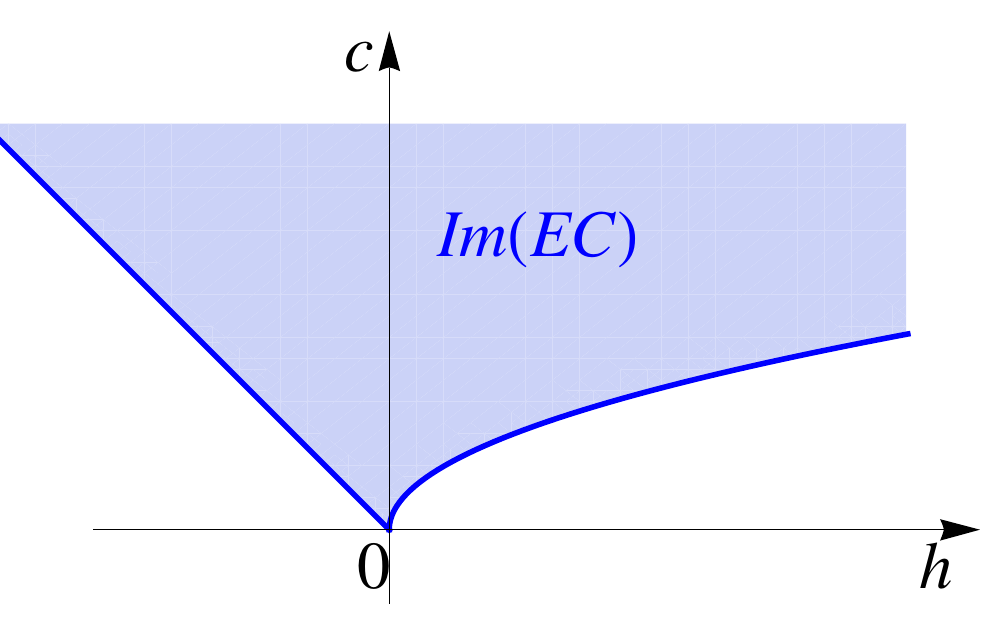}
\caption{The image of the energy-Casimir mapping.}
\label{fig:1}\end{figure}

A point $(x_0,y_0,z_0)\in\mathbb{R}^3$ is a critical point of the energy-Casimir mapping if the rank of the Jacobian matrix of $\mathcal{EC}$ at this point is less than 2.
\begin{proposition}\label{p5}
The critical points of the energy-Casimir mapping \eqref{EC} are given by the following families
\begin{equation}\label{echilibre}
E_1(M,0,0),E_2(0,M,0),E_3(0,0,M),E_4(M,M,0),E_5(M,-M,0),M\in\mathbb{R}.
\end{equation}
\end{proposition}
\begin{proof}
We have $$D\mathcal{EC}(x,y,z)=\left[\begin{array}{c}
DH(x,y,z)\\DC(x,y,z)
\end{array}\right]=\left[\begin{array}{ccc}
x^3&y^3&-z\\x&y&z
\end{array}\right].$$
Imposing the condition rank$\,D\mathcal{EC}<2$, the conclusion follows.
\end{proof}
Now we determine the images of these critical points through the energy-Casimir mapping. We have
\begin{align*}
&\mathcal{EC}(E_1)=\mathcal{EC}(E_2)=\left(\frac{1}{4}M^4,\frac{1}{2}M^2\right)=(h,c)~,~c=\sqrt{h}~,~h\geq 0,\\
&\mathcal{EC}(E_3)=\left(-\frac{1}{2}M^2,\frac{1}{2}M^2\right)=(h,c)~,~~c=-h~,~h\leq 0,\\
&\mathcal{EC}(E_4)=\mathcal{EC}(E_5)=\left(\frac{1}{2}M^4,M^2\right)=(h,c)~,~c=\sqrt{2h}~,~h\geq 0.
\end{align*} 
The images through the energy-Casimir mapping of its critical points are given by the curves (see Figure \ref{fig:2})
\begin{align*}
&\Sigma_{1,2}^s=\{(h,c):c=\sqrt{h}~,~h\geq 0\},\\
&\Sigma_{3}^s=\{(h,c):c=-h~,~h\leq 0\},\\
&\Sigma_{4,5}^u=\{(h,c):c=\sqrt{2h}~,~h>0\}.
\end{align*} 
We also consider the sets
\begin{align*}
&\Sigma_{p}^1=\{(h,c):\sqrt{h}<c<\sqrt{2h}~,~h> 0\},\\
&\Sigma_{p}^2=\{(h,c):c>-h~,~h<0\}\cup\{(h,c):c>\sqrt{2h}~,~h>0\}.
\end{align*}
\begin{remark}\label{r2}
The images of the critical points through the energy-Casimir mapping lead to the following partition of the image of the energy-Casimir mapping
\begin{equation}\label{partition}
\mbox{Im}(\mathcal{EC})=\Sigma_{1,2}^s\cup\Sigma_{p}^1\cup\Sigma_{4,5}^u\cup\Sigma_{p}^2\cup\Sigma_{3}^s.
\end{equation}
Note that there is only one bifurcation point in this partition, namely $(0,0)$.
\end{remark}
\begin{figure}[h!]
\centering
\includegraphics[width=0.5\linewidth]{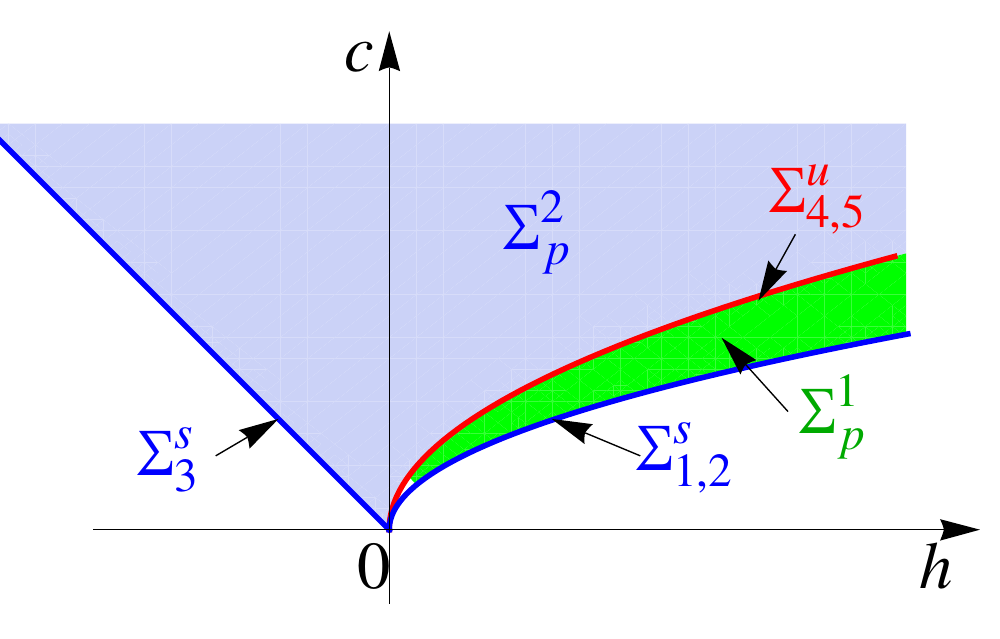}
\caption{The semialgebraic partition of the image of the energy-Casimir mapping given by the critical points.}
\label{fig:2}\end{figure}
As it has been reported in the above-mentioned papers (Remark \ref{r1}), there are some connections between the partition of the image of the energy-Casimir mapping and the dynamics of the corresponding system. More precisely, in the case when $\mbox{Im}(\mathcal{EC})$ is a proper convex subset of $\mathbb{R}^2$ its boundary is the union of the images of the nonlinearly stable equilibrium points through $\mathcal{EC}$. Moreover, if a curve that gives the partition and belongs to the interior of $\mbox{Im}(\mathcal{EC})$ do not have bifurcation points, then it is given by the images of the unstable equilibrium points through $\mathcal{EC}$. Furthermore, if such a curve is an arc of parabola, then homoclinic orbits were computed.  
In addition, the open subsets $\Sigma_p$ of $\mbox{Im}(\mathcal{EC})$ are related to periodic orbits.

In our case $\mbox{Im}(\mathcal{EC})$ is a non-convex set. Therefore it is natural to ask whether these properties remain true, namely the critical points $E_1,E_2,E_3$ are stable equilibrium points and $E_4,E_5$ are unstable, and also there are periodic orbits in the considered dynamics. Moreover, are there homoclinic orbits? 

In next sections we give answers to these questions.
\section{Stability}
In this section we study the stability of the equilibrium points of system \eqref{1}. We use Arnold stability test \cite{Arnold65}, Lyapunov functions, and First Lyapunov's Stability Criterion \cite{Lyapunov49}.   

It is easy to see that system \eqref{1} takes the form $\dot{\bf x}=\nabla H\times\nabla C$, ${\bf x}=(x,y,z).$ Therefore the equilibrium points of system \eqref{1} are in fact the critical points \eqref{echilibre} of the energy-Casimir mapping \eqref{EC}.

In the next proposition we give the results regarding the stability of the equilibrium points  $$E_1(M,0,0), E_2(0,M,0), E_3(0,0,M), E_4(M,M,0),E_5(M,-M,0), M\in\mathbb{R}.$$
\begin{proposition} 
a) The points $E_1,E_2,E_3$ are nonlinearly stable equilibrium points for every $M\in\mathbb{R}$.\\
b) The equilibrium points $E_4,E_5$ are unstable for every $M\in\mathbb{R},M\not=0$.
\end{proposition}
\begin{proof}
a) Let $M\in\mathbb{R},M\not=0$ and the equilibrium point $E_1(M,0,0)$. We consider the function $F=H+\lambda C$. The condition $\nabla F(M,0,0)={\bf 0}$ leads to $\lambda=-M^2$. Using the fact that $dC(M,0,0)=0$, we obtain $d^2F(M,0,0)=-M^2dy^2-(M^2+1)dz^2$ that is negative definite. By Arnold stability test we deduce that the equilibrium point $E_1$ is nonlinearly stable. We analogously proceed for the equilibrium points $E_2$ and $E_3$.

If $M=0$, then all the equilibrium points coincide. In this case the Casimir $C(x,y,z)=\frac{1}{2}x^2+\frac{1}{2}y^2+\frac{1}{2}z^2$ is a Lyapunov function for the equilibrium point $(0,0,0)$ and $\frac{dC}{dt}=0$. Therefore the equilibrium point $(0,0,0)$ is nonlinearly stable.\\
b) Let $J(x,y,z)$ be the matrix of linear part of system \eqref{1},
 that is $$J(x,y,z)=\!\left[\begin{array}{ccc}
  0& 3y^2z+z&y^3+y\\
 -3x^2z-z&0&-x^3-x\\3x^2y-y^3&x^3-3xy^2&0
\end{array}\right].$$
The characteristic roots of $J(E_4)$ and $J(E_5)$ are given by
$$\lambda_{1}=0,\ \ \lambda_{2,3}=\pm 2M^2\sqrt{M^2+1}.$$
Therefore, for every $M\in\mathbb{R},M\not=0$ there is a positive eigenvalue and consequently the equilibrium points $E_4$ and $E_5$ are unstable. 
\end{proof}
\begin{remark}
The images of the nonlinearly stable equilibrium points through the energy-Casimir mapping are the curves $\Sigma_{1,2}^s$ and $\Sigma_3^s$, where the superscript ``s'' means stable, as in above-mentioned papers. Moreover, the set $\Sigma_{4,5}^u$ is the image of the unstable equilibrium points through $\mathcal{EC}$.
\end{remark}

\section{Fibers of the energy-Casimir mapping}
The fiber of the energy-Casimir mapping $\mathcal{EC}$ corresponding to an element $(h_0,c_0)\in\mbox{Im}(\mathcal{EC})$ is the set
\begin{equation}\label{F}
\mathcal{F}_{(h_0,c_0)}=\left\{(x,y,z)\in\mathbb{R}^3~|~\mathcal{EC}(x,y,z)=(h_0,c_0)\right\}.
\end{equation}
The implicit equation of the above fiber is given by
\begin{equation}\label{IF}
\mathcal{F}_{(h_0,c_0)}:\left\{\begin{array}{l}
H(x,y,z)=h_0\vspace{4pt}\\C(x,y,z)=c_0~.
\end{array}\right.
\end{equation}
On the other hand the dynamics of the considered system takes place at the intersection of the level sets $H(x,y,z)=constant$, $C(x,y,z)=constant$. Therefore an orbit of the our system is given implicitly by the above fiber.

Taking into account the partition \eqref{partition} of $\mbox{Im}(\mathcal{EC})$, in this section we point out the topology of the fibers of energy-Casimir mapping \eqref{EC}. We prove the existence of the periodic orbits. Using numerical integration, we emphasize a possible existence of some heteroclinic cycles.

Let $M>0$. Fixing $c=c_0>0$ and varying $h$ such that $-c_0\leq h\leq c_0^2$, the straight-line of equation $c=c_0$ intersects all the sets of the partition \eqref{partition} of the image of the energy-Casimir mapping. The intersections of the level sets $H(x,y,z)=h$ and $C(x,y,z)=c_0$ when $(h,c_0)$ belongs to $\Sigma_{1,2}^s$, $\Sigma_{p}^1$, $\Sigma_{4,5}^u$, $\Sigma_{p}^2$, and $\Sigma_{3}^s$ are presented in Figure \ref{fig:3} (a), (b), (c), (d)-(e), and (f)  respectively. We notice that around the equilibrium points $E_1(\pm M,0,0)$ and $E_2(0,\pm M,0)$ (Figure \ref{fig:3} (a)) there are four families of periodic orbits (Figure \ref{fig:3} (b)) which collide (Figure \ref{fig:3} (c)) when $h$ decreases and takes the value $\frac{1}{2}c_0^2$. Four orbits are obtained and they are contained in the intersection of the level sets $H(x,y,z)=h$ and $C(x,y,z)=c_0$, where $(h,c_0)\in\Sigma_{4,5}^u$. Also, in this case, the fiber $\mathcal{F}_{(h,c)}$ contains the unstable equilibrium points $E_4$ and $E_5$. Afterwards these orbits split in two families of periodic orbits around $E_3(0,0,\pm M)$ (Figure \ref{fig:3} (d),(e)) which tend to $E_3$ as $h\to -c_0$ (Figure \ref{fig:3} (f)). Therefore we deduce that apparently the above-mentioned four orbits are heteroclinic orbits.
 
\begin{figure}[h]
\centering
\subfigure[$(h,c)\in\Sigma_{1,2}^s$]{\label{fig:3a}\includegraphics[width=.3\linewidth]{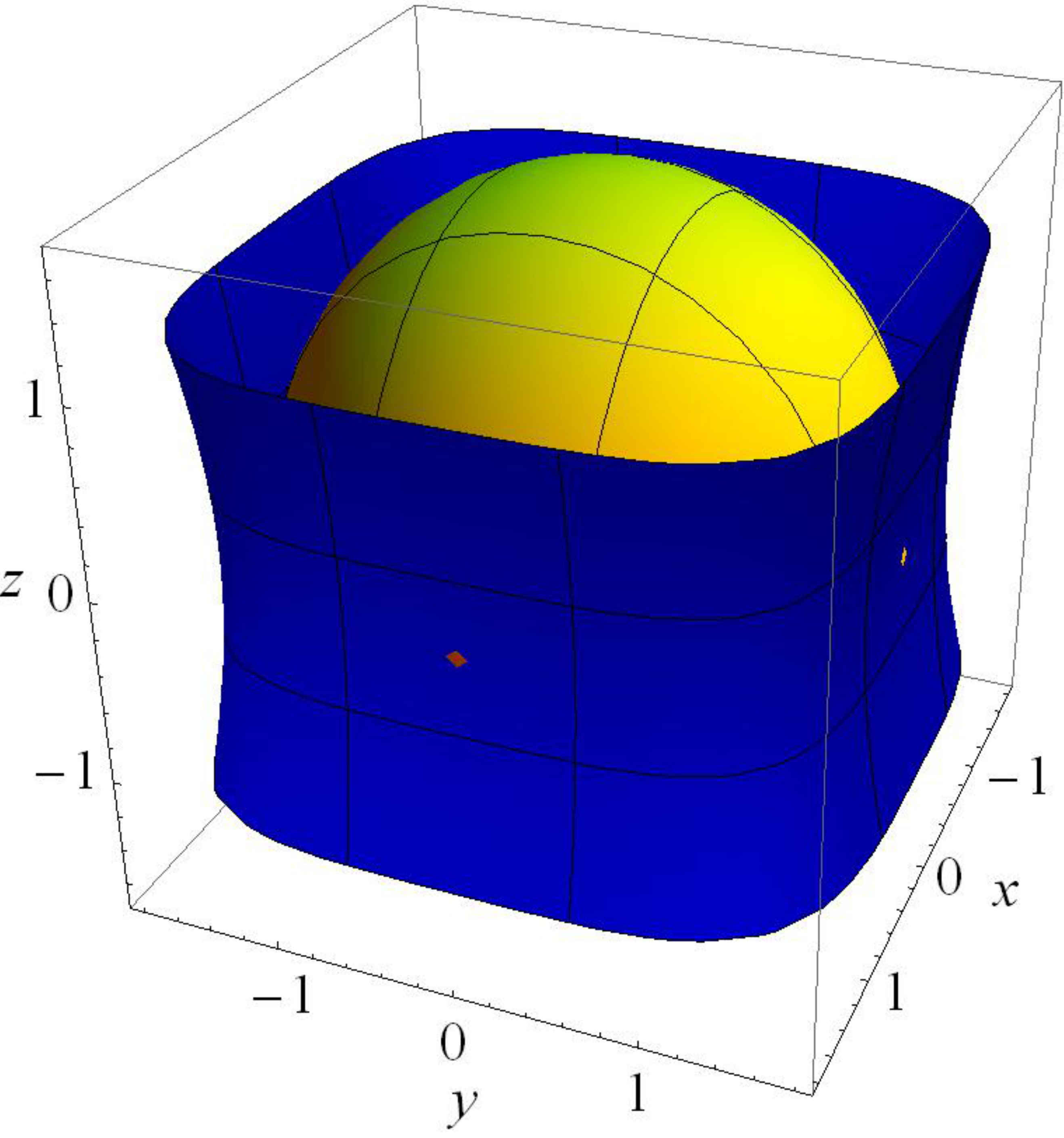}}\hspace{0.3cm}
\subfigure[$(h,c)\in\Sigma_{p}^1$]{\label{fig:3b}\includegraphics[width=.3\linewidth]{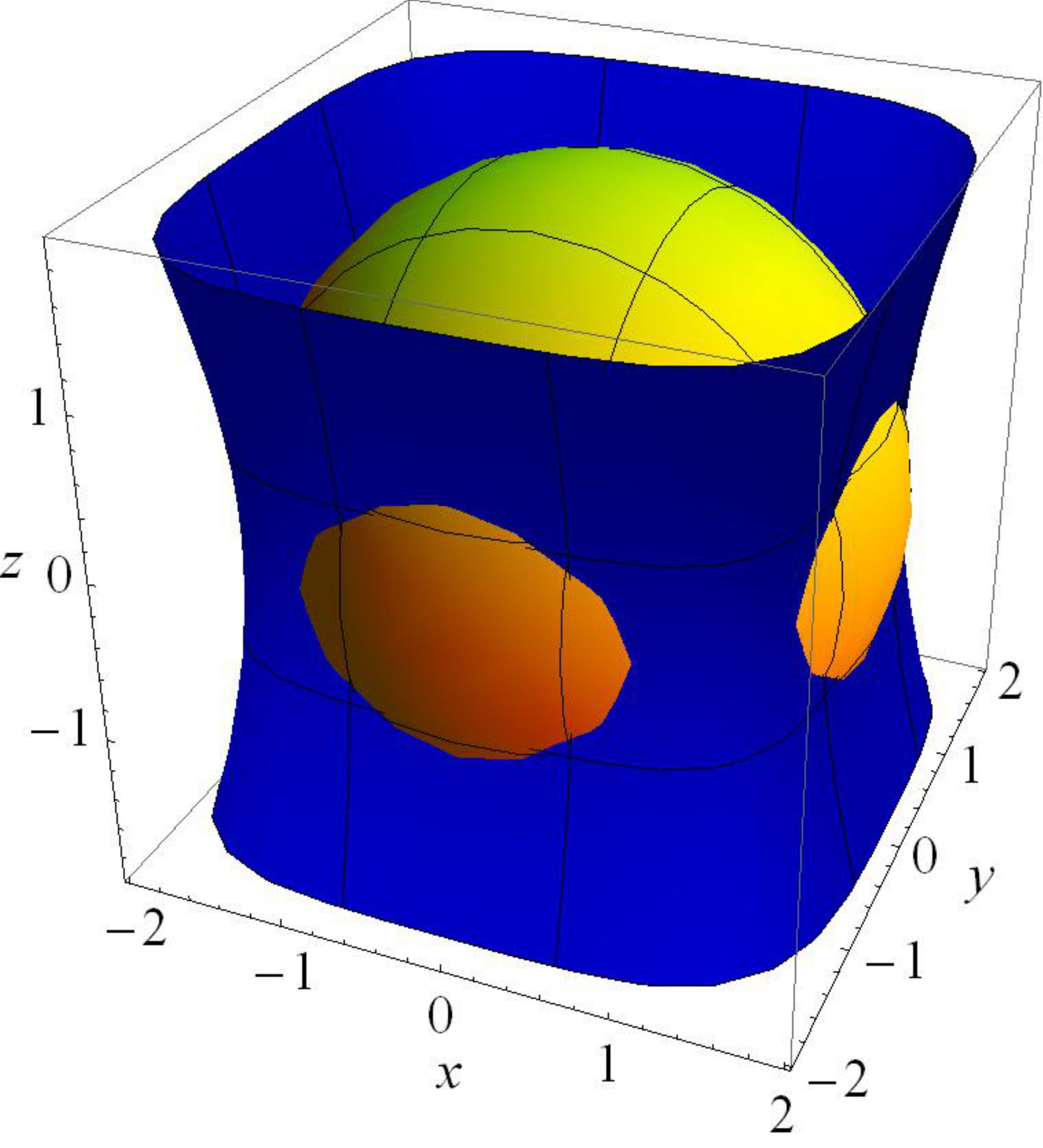}}\hspace{0.3cm}
\subfigure[$(h,c)\in\Sigma_{4,5}^u$]{\label{fig:3c}\includegraphics[width=0.3\linewidth]{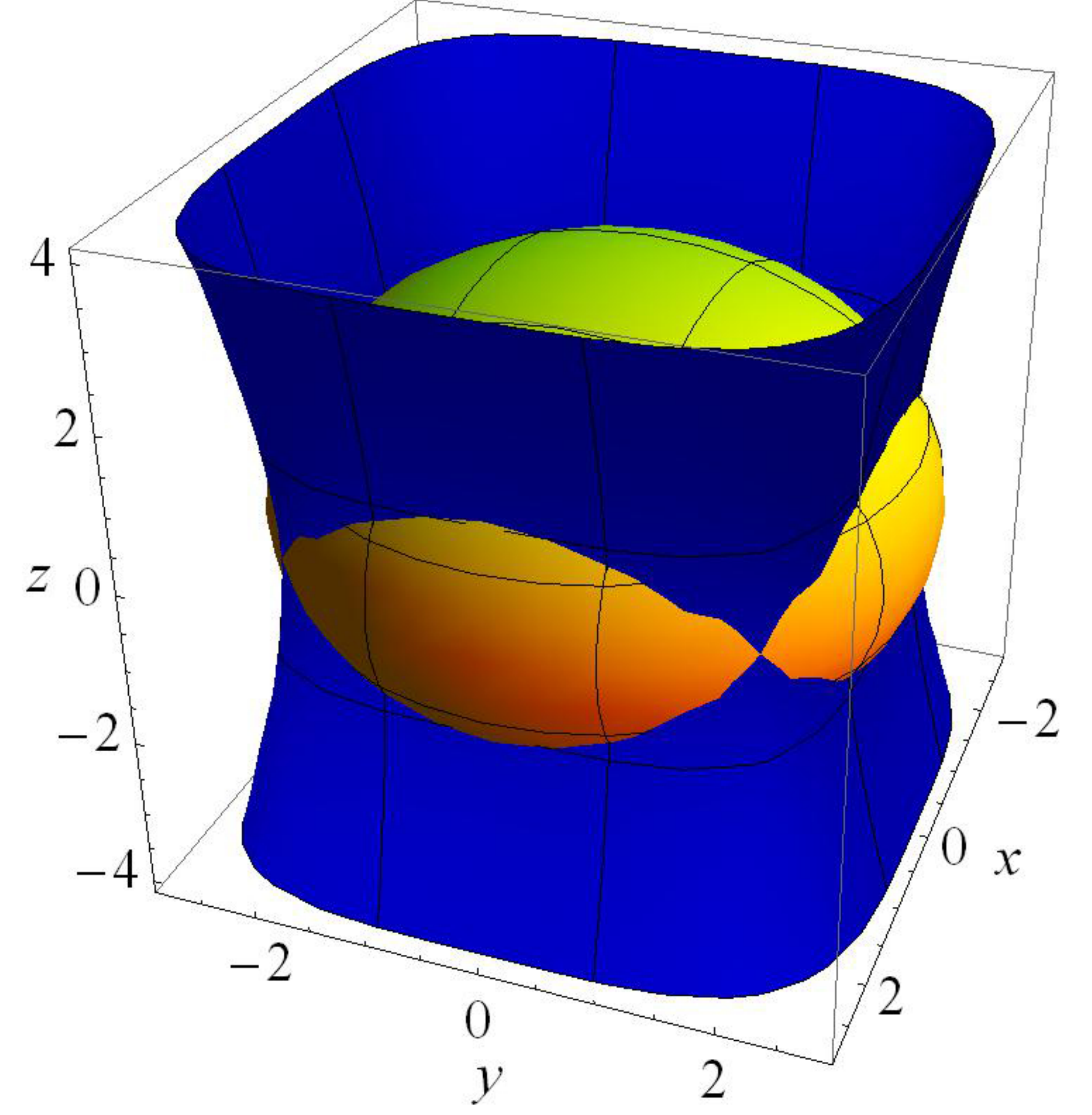}}
\subfigure[$(h,c)\in\Sigma_{p}^2,h\geq 0$]{\label{fig:3d}\includegraphics[width=.3\linewidth]{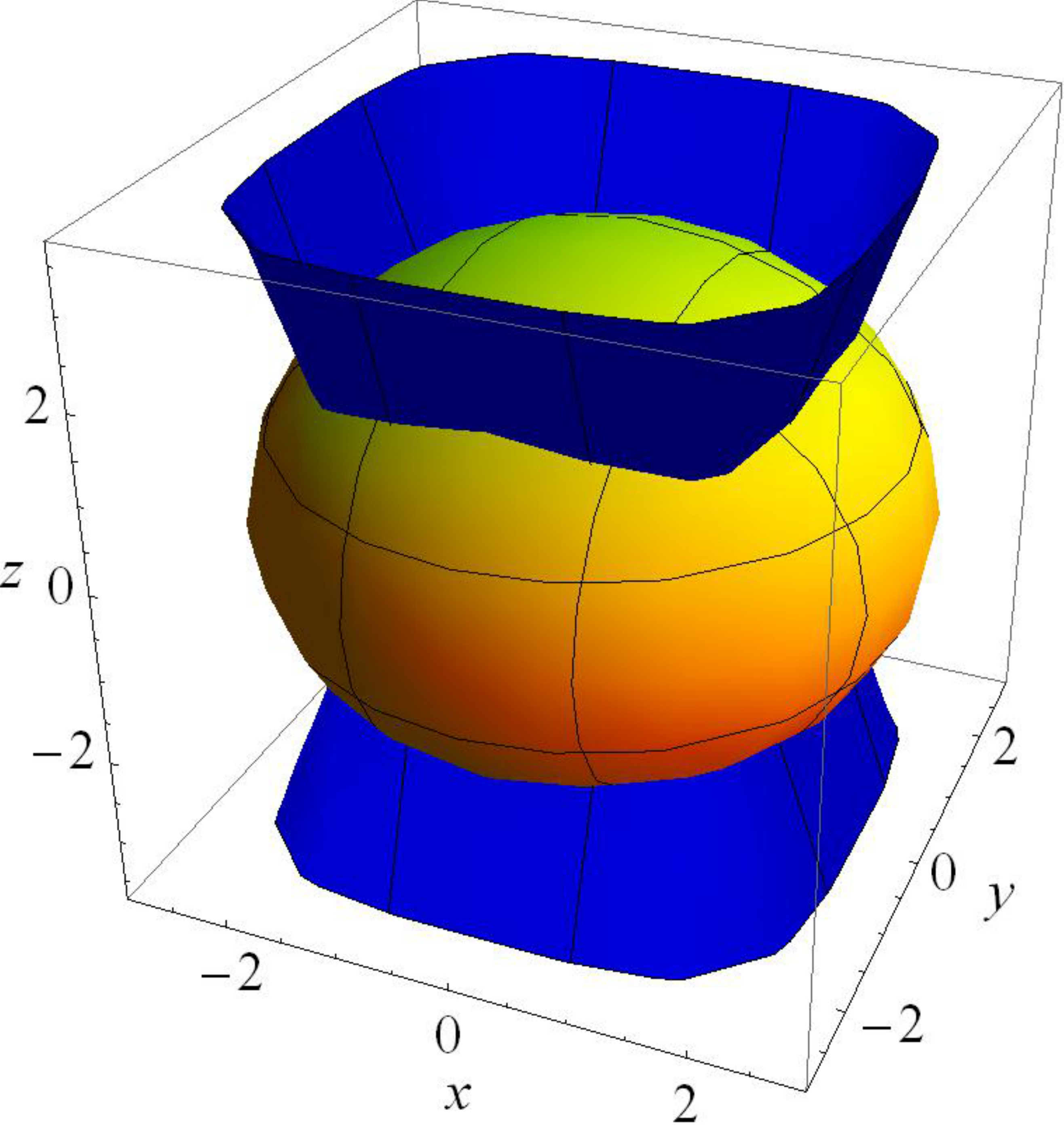}}
\hspace{0.3cm}
\subfigure[$(h,c)\in\Sigma_{p}^2,h<0$]{\label{fig:3e}\includegraphics[width=.3\linewidth]{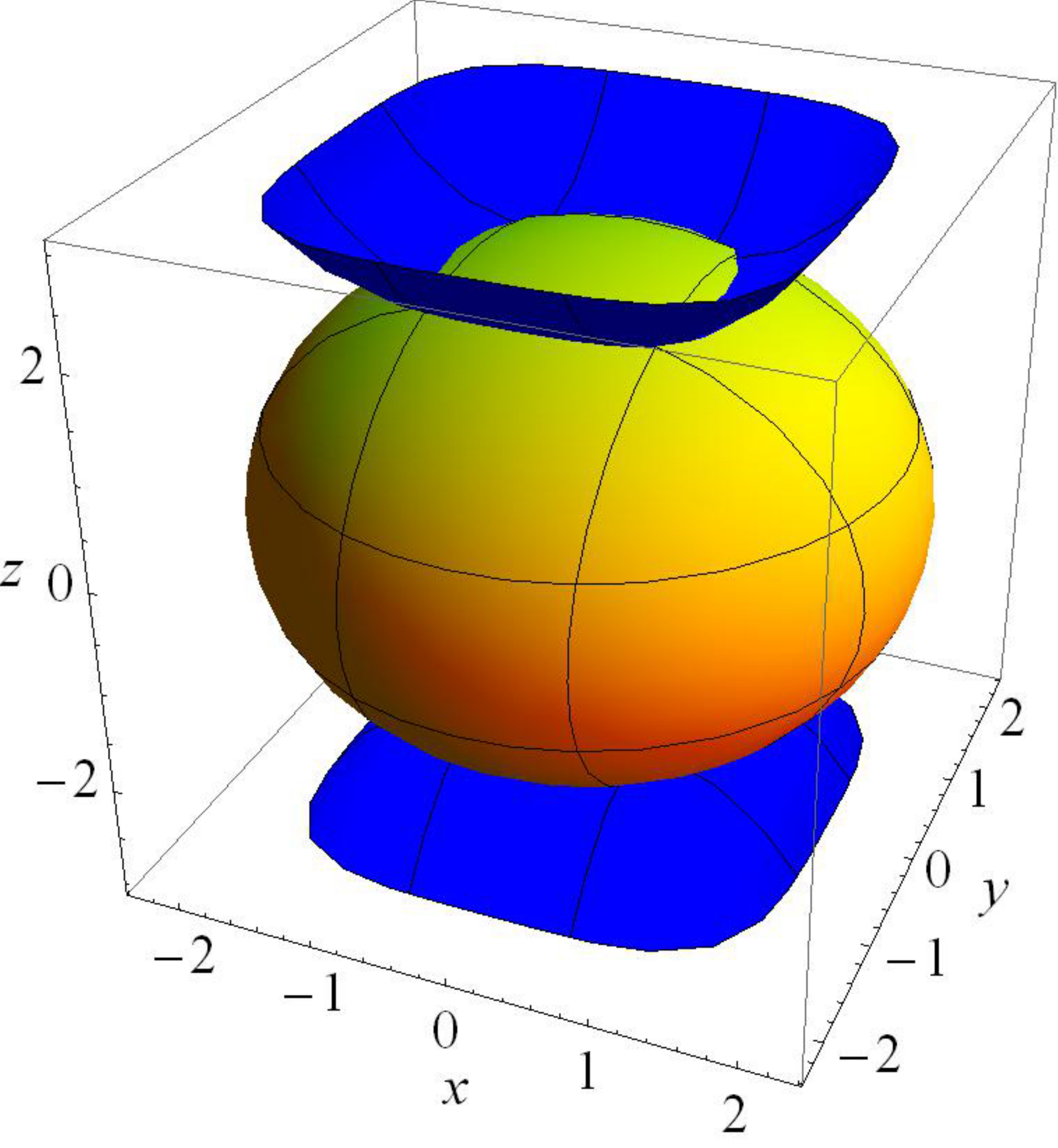}}\hspace{0.3cm}
\subfigure[$(h,c)\in\Sigma_{3}^s$]{\label{fig:3f}\includegraphics[width=.3\linewidth]{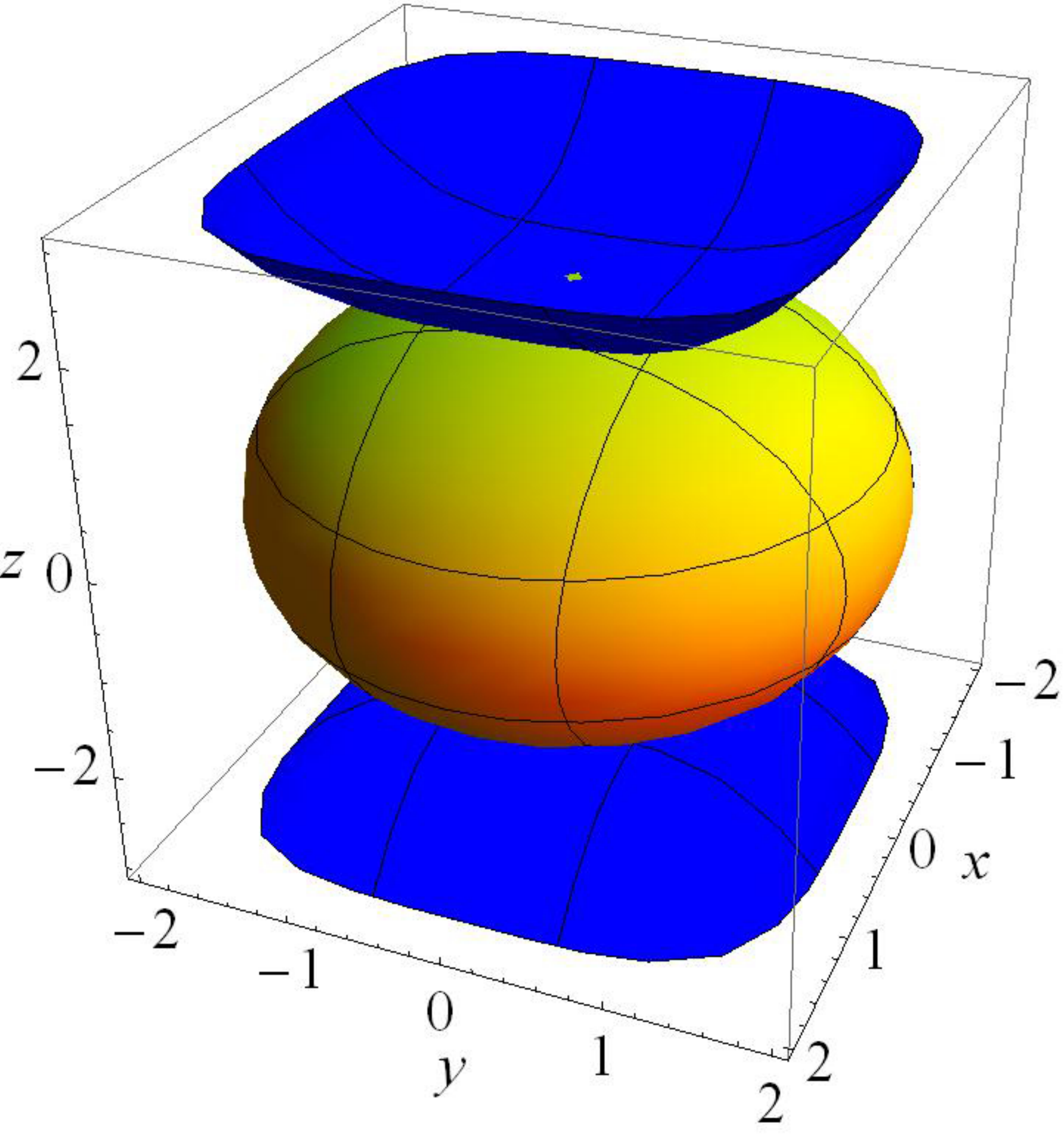}}
\caption{Intersections of the level sets $H(x,y,z)=h$ and $C(x,y,z)=c$, namely 
(a) and (f): stable equilibrium points; (b),(d),(e): periodic orbits; (c): heteroclinic orbits.}
\label{fig:3}
\end{figure}

We begin our study with the pairs $(h,c)$ that belong to the boundary of $\mbox{Im}(\mathcal{EC})$. 
\begin{proposition}
Let $(h,c)\in\Sigma_{1,2}^s\cup\Sigma_3^s$. Then the corresponding fiber contains only nonlinear stable equilibrium points, namely:\\
a) If $(h,c)\in\Sigma_{1,2}^s$, then $\mathcal{F}_{(h,c)}=\{(\pm\sqrt{2c},0,0)\}\cup\{(0,\pm\sqrt{2c},0)\}$ (Figure \ref{fig:3} (a));\vspace{4pt}\\
b) If $(h,c)\in\Sigma_3^s$, then $\mathcal{F}_{(h,c)}=\{(0,0,\pm\sqrt{2c})\}$ (Figure \ref{fig:3} (f));\vspace{4pt}\\
c) If $(h,c)\in\Sigma_{1,2}^s\cap\Sigma_3^s$, then $\mathcal{F}_{(h,c)}=\{(0,0,0)\}$.
\end{proposition}
\begin{proof}
a) By hypothesis, $c=\sqrt{h}$. Using \eqref{F} and \eqref{H-C}, we obtain $$z^2(z^2+2+2x^2+2y^2)+2x^2y^2=0.$$ Hence $z=0$ and $xy=0$. Therefore the conclusion follows.

Analogously we prove b) and c).
\end{proof}

We study the others cases in the next subsections.

\subsection{Periodic orbits}
Let $(h,c)\in\Sigma_{p}^1\cup\Sigma_p^2$, where $\Sigma_{p}^1$ and $\Sigma_{p}^2$ are given by \eqref{partition}. The intersection of the level sets $H(x,y,z)=h$ and $C(x,y,z)=c$ suggests the existence of the periodic orbits (see Figure \ref{fig:3} (b),(d),(e)). We prove the existence of these periodic orbits around the nonlinear  stable equilibrium points using a version of Moser's theorem in the case of zero eigenvalue \cite{BiPuTu07}.

\begin{proposition}\label{prop9a}
Let $E_1=\left(M,0,0\right)$ be a nonlinear stable equilibrium point of system (\ref{1}) such that $M\in\mathbb{R}^*$. Then for each sufficiently small
$\varepsilon\in\mathbb{R}_+^*$, any integral surface
$$
\Sigma_\varepsilon^{E_1}\,:\,~-\frac{1}{4}(x^4+y^4)+\frac{M^2}{2}(x^2+y^2)+\frac{M^2+1}{2}z^2-\frac{1}{4}M^4=\varepsilon^2
$$
contains at least one periodic orbit $\gamma_\varepsilon^{E_1}$ of system (\ref{1}) whose period is close to $\frac{2\pi}{M^2\sqrt{M^2+1}}$.
\end{proposition}
\begin{proof} 
We apply Theorem 2.1 from \cite{BiPuTu07}. 

The characteristic polynomial associated with the linearization of system (\ref{1}) at $E_1$ has the eigenvalues $\lambda_{1}=0$ and $\lambda_{2,3}=\pm iM^2\sqrt{M^2+1}$. Furthermore, the eigenspace corresponding to the eigenvalue zero is $\mbox{span}_{\mathbb{R} }\left\{(1,0,0)\right\}$. The constant of motion of system \eqref{1} given by 
$$I(x,y,z)=-\frac{1}{4}(x^4+y^4)+\frac{M^2}{2}(x^2+y^2+z^2)$$ has the properties: $dI(M,0,0)=0$ and $\left.d^2I\left(M,0,0\right)\right|_{W\times W}=M^2dy^2+(M^2+1)dz^2>0$, where $W=\ker dC\left(M,0,0\right)=\mbox{span}_{\mathbb{R} }\left\{(0,1,0),(0,0,1)\right\}.$

Therefore the conclusion follows via Theorem 2.1 \cite{BiPuTu07}.
\end{proof}
We obtain the same result for $E_2$. Analogously we get the next result.
\begin{proposition}\label{prop9b}
Let $E_3=\left(0,0,M\right)$ be a nonlinear stable equilibrium point of system (\ref{1}) such that $M\in\mathbb{R}^*$. Then for each sufficiently small
$\varepsilon\in\mathbb{R}_+^*$, any integral surface
$$
\Sigma_\varepsilon^{E_1}\,:\,~\frac{1}{4}(x^4+y^4)+\frac{1}{2}(x^2+y^2)=\varepsilon^2
$$
contains at least one periodic orbit $\gamma_\varepsilon^{E_3}$ of system (\ref{1}) whose period is close to $\frac{2\pi}{|M|}$.
\end{proposition}

\subsection{Numerical integration. Heteroclinic orbits}
Consider $(h,c)\in\Sigma_{4,5}^u$, that is $h>0$ and $c=\sqrt{2h}.$ The fiber $\mathcal{F}_{(h,c)}$ contains the unstable equilibrium points $E_4$ and $E_5$. Its implicit equation is given by 
\begin{equation}\label{HO}
\dfrac{1}{4}x^4+\dfrac{1}{4}y^4-\dfrac{1}{2}z^2=\dfrac{c^2}{2}~,~~\dfrac{1}{2}x^2+\dfrac{1}{2}y^2+\dfrac{1}{2}z^2=c.
\end{equation}
The intersection of the above level sets is shown in Figure \ref{fig:3} (c) and it suggests the existence of four pair of heteroclinic orbits that connect the unstable equilibrium points $E_4(\pm \sqrt{c},\pm \sqrt{c},0)$ and $E_5(\pm\sqrt{c},\mp\sqrt{c},0)$.

We recall that a heteroclinic orbit $\mathcal{HE}:\mathbb{R}\to\mathbb{R}^3$ is a solution $(x(t),y(t),z(t))$ of the considered system that connects two unstable equilibrium points $e_1$ and $e_2$ of the system, that is $\mathcal{HE}(t):=(x(t),y(t),z(t))$ and $\mathcal{HE}(t)\to e_1$ as $t\to -\infty$, $\mathcal{HE}(t)\to e_2$ as $t\to \infty$.

We give the numerical simulation of these heteroclinic orbits applying the mid-point rule (see \cite{Austin} and references therein) to system \eqref{1}.  

Consider the Hamilton-Poisson realization of system \eqref{1} given by Proposition \ref{p1}:
$$\dot{\bf x}=\Pi_1({\bf x})\nabla H({\bf x})~, ~~{\bf x}=(x,y,z)^t,$$
where $\Pi_1$ is the Poisson structure \eqref{pi1} and $H$ is the Hamiltonian function. The mid-point rule is given by the following implicit recursion \cite{Austin}
$$\frac{{\bf x}_{k+1}-{\bf x}_k}{\Delta t}=\Pi_1\left(\frac{{\bf x}_{k}+{\bf x}_{k+1}}{2}\right)\nabla H\left(\frac{{\bf x}_{k}+{\bf x}_{k+1}}{2}\right),$$
where $\Delta t$ is the time-step. ``If $\Pi({\bf x})$ is linear in ${\bf x}$, then the mid-point rule is an almost Poisson integrator, that is it preserve the Poisson structure up to second order'' \cite{Austin}. Furthermore, ``the mid-point rule preserves exactly any conserved quantity having only linear and quadratic terms'' \cite{Austin}.

The integrator for system \eqref{1} is given by
\begin{align}\label{I}
&\frac{x_{k+1}-x_{k}}{\Delta t}=\frac{1}{16}(y_{k}+y_{k+1})(z_{k}+z_{k+1})(4+(y_{k}+y_{k+1})^2)\nonumber\\
&\frac{y_{k+1}-y_{k}}{\Delta t}=-\frac{1}{16}(x_{k}+x_{k+1})(z_{k}+z_{k+1})(4+(x_{k}+x_{k+1})^2)\\
&\frac{z_{k+1}-z_{k}}{\Delta t}=\frac{1}{16}(x_{k}+x_{k+1})(y_{k}+y_{k+1})((x_{k}+x_{k+1})^2-(y_{k}+y_{k+1})^2)\nonumber
\end{align}
\begin{remark}{}
Because the Poisson bracket \eqref{pi1} is linear, the mid-point rule of system \eqref{1} given by \eqref{I} is an almost Poisson integrator. Moreover, the   Casimir $C$ \eqref{H-C} is quadratic, hence it is preserved by this integrator. 
\end{remark}

We implemented algorithm \eqref{I} in Wolfram Mathematica$^{TM}$. 

First, we fix $h=0.5$, $c=\sqrt{2h}=1$, and $z_1=0.5$ and compute $x_1$, $y_1$ such that $H(x_1,y_1,z_1)=h$ and $C(x_1,y_1,z_1)=c$. We find eight solutions and we take $x_1=1.25338$ and $y_1=0.42312$. Choosing the time-step $\Delta t=0.015$, after 160 iterations we get the point $(1.00305, -0.996944, 0.00128394)$ which is closer to the unstable equilibrium point $E_5(1,-1,0)$. To simulate the behavior of the orbit when $t$ decreases to $-\infty$, we consider the same initial point $(x_1,y_1,z_1)$ and $\Delta t=-0.015$. After 160 iterations we get the point $(1.00438, 0.995591, -0.00465251)$ which is closer to the unstable equilibrium point $E_4(1,1,0)$.  The discrete orbit is shown in Figure \ref{fig:4} (a). We remark that the points $(x_k,y_k,z_k)$ are very close near the both unstable equilibrium points. Analogously we obtain a second orbit that connects the points $E_4(1,1,0)$ and $E_5(1,-1,0)$ (Figure \ref{fig:4} (b)). 
\begin{figure}[h]
\centering
\subfigure[]{\label{fig:4a}\includegraphics[width=.45\linewidth]{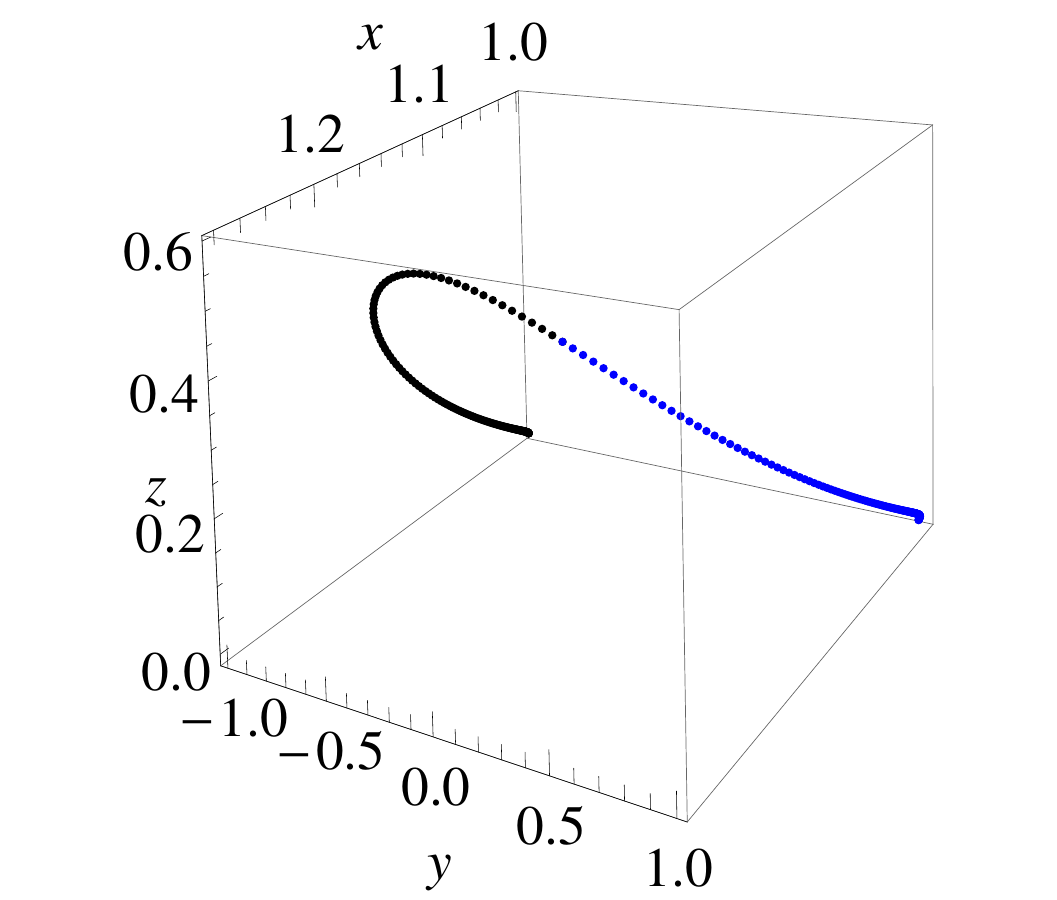}}\hspace{0.3cm}
\subfigure[]{\label{fig:4b}\includegraphics[width=.4\linewidth]{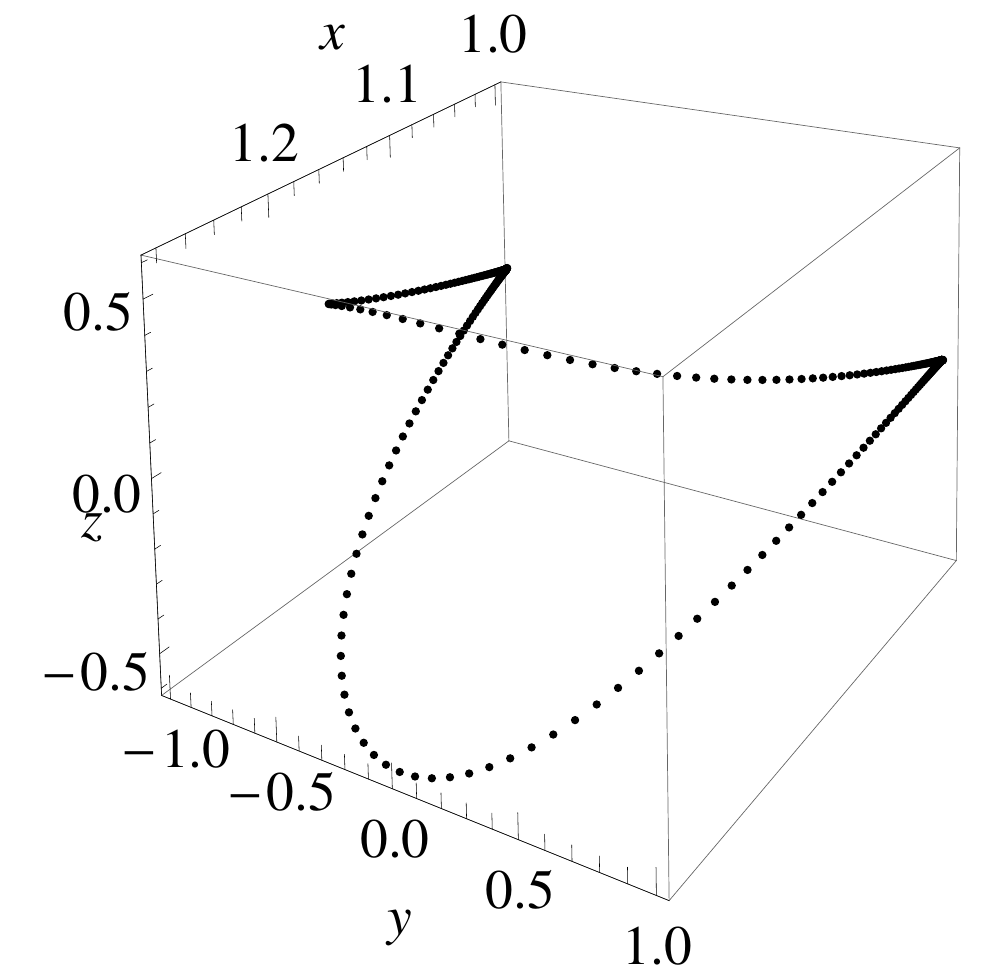}}\hspace{0.3cm}
\caption{Numerical simulation of heteroclinic orbits:  
(a) an orbit; (b) a pair of orbits.}
\label{fig:4}
\end{figure}

In Figure \ref{fig:5} the obtained pair of orbits is shown together with the intersection of the level sets $H(x,y,z)=\frac{1}{2}$ and $C(x,y,z)=1$ which correspond to the equilibrium points $E_4(1,1,0)$, $E_5(1,-1,0)$. We notice a very well superposition of these curves.
\begin{figure}[h!]
\centering
\includegraphics[width=0.5\linewidth]{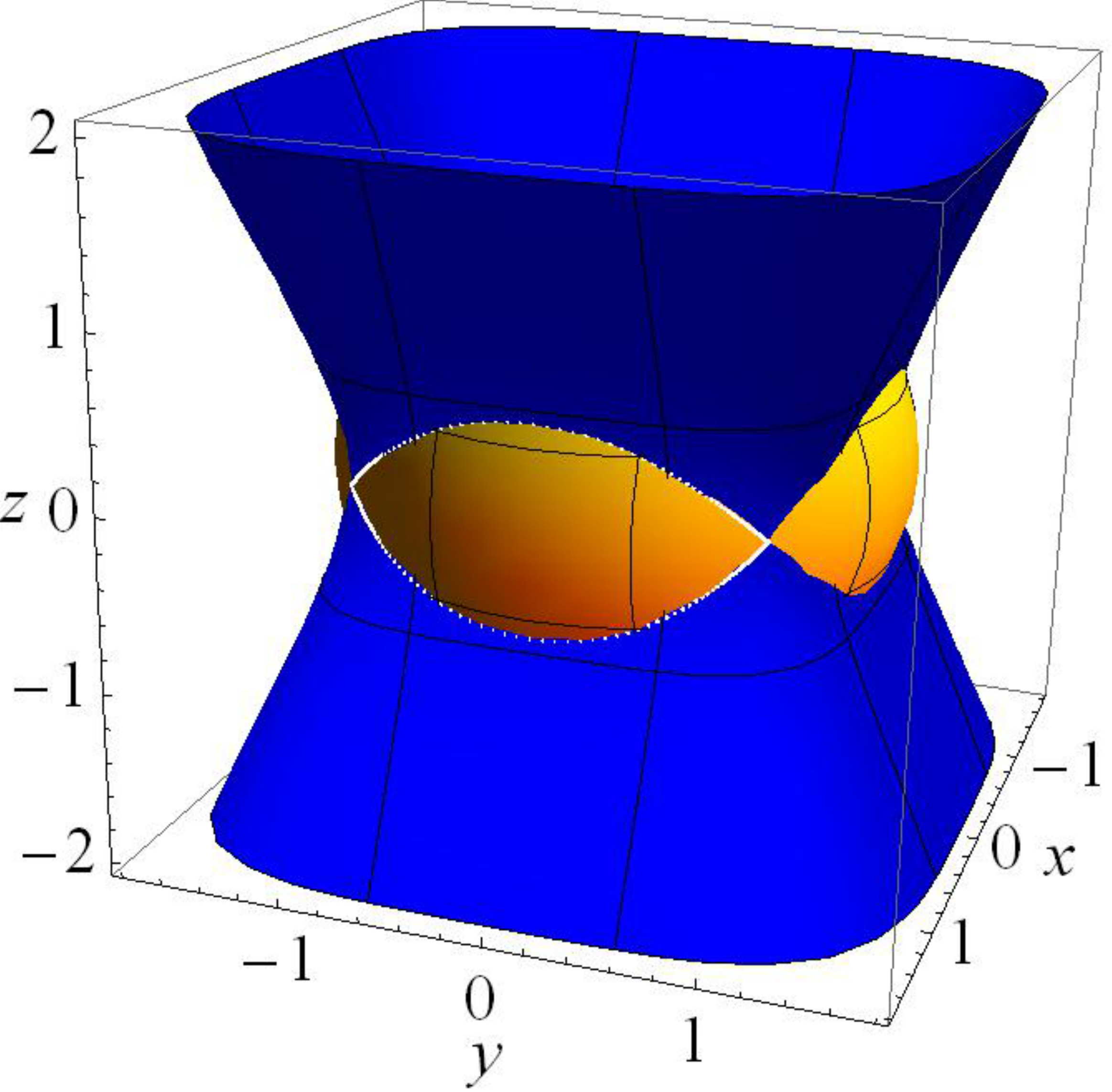}
\caption{A pair of heteroclinic orbits and the intersection of the level sets $H=c^2/2$, $C=c$.}
\label{fig:5}\end{figure}

In the same manner we obtain the others ``heteroclinic'' orbits (Figure \ref{fig:6} (a)). They arise naturally taking into account the symmetries of system \eqref{1} given by the transformations  $(x,y,z)\to (-x,-y,z)$,  $(x,y,z)\to (-x,y,-z)$,  $(x,y,z)\to (x,-y,-z)$,  $(x,y,z)\to (-y,x,z)$,  $(x,y,z)\to (y,-x,z)$. We also remark that the unstable equilibrium points $E_4(1,1,0)$, $E_5(-1,1,0)$, $E_4(-1,-1,0)$, and $E_5(1,-1,0)$ are connected by two cycles of ``heteroclinic'' orbits. Such a cycle is shown in Figure \ref{fig:6} (b).   
\begin{figure}[!h]
\centering
\subfigure[]{\label{fig:6a}\includegraphics[width=.45\linewidth]{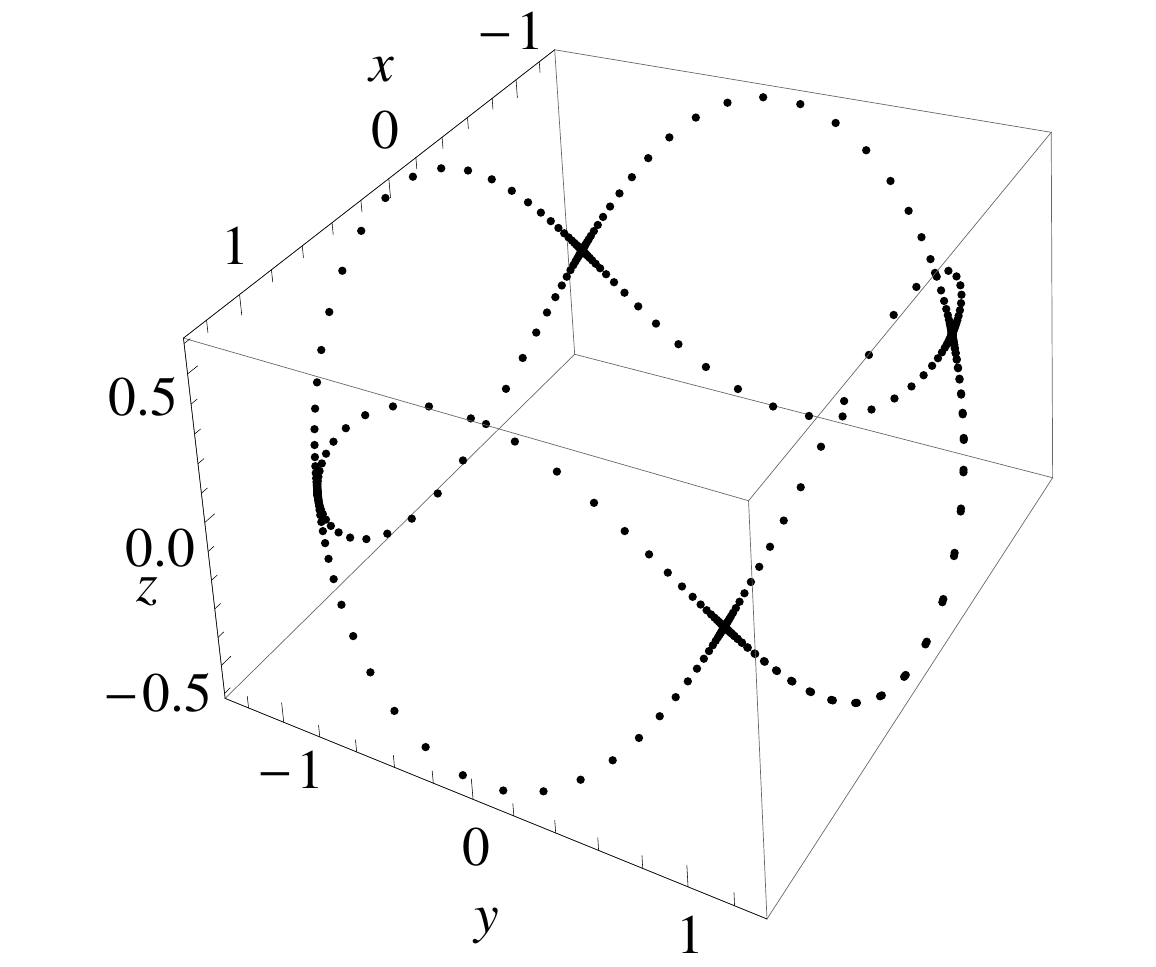}}\hspace{0.3cm}
\subfigure[]{\label{fig:6b}\includegraphics[width=.4\linewidth]{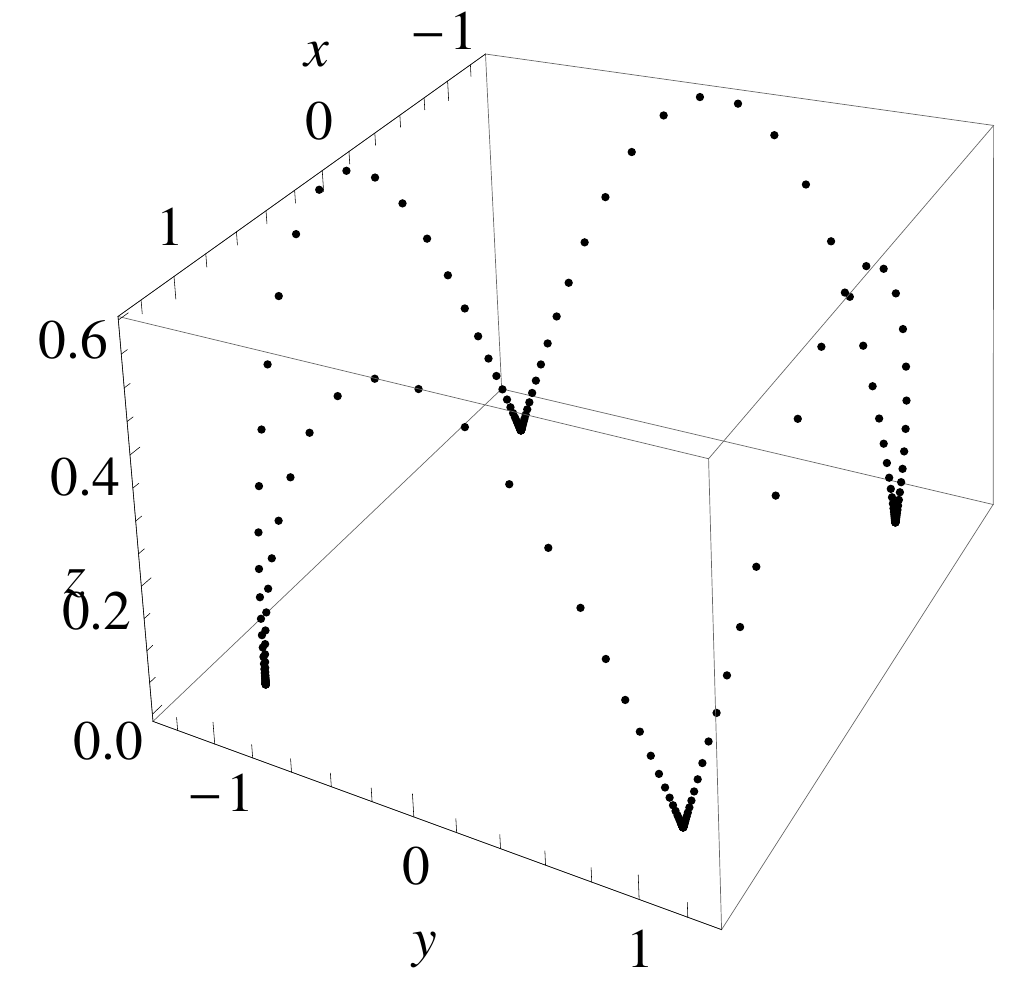}}\hspace{0.3cm}
\caption{Numerical simulation of heteroclinic orbits:  
(a) four pairs of orbits; (b) a cycle of orbits.}
\label{fig:6}
\end{figure}
\section{ACKNOWLEDGMENTS}
This work was supported by research grants PCD-TC-2017.

\smallskip
{\footnotesize
\hspace*{0.5cm}
\begin{minipage}[t]{8cm}$$\begin{array}{l}
\mbox{Cristian L\u azureanu -- Department of Mathematics,}\\
\mbox{Politehnica University of Timi\c soara},\\
 \mbox{P-ta Victoriei 2, 300 006, Timi\c soara, ROMANIA}\\
\mbox{E-mail: cristian.lazureanu@upt.ro}\end{array}$$
\end{minipage}\\
{\footnotesize
\hspace*{0.5cm}
\begin{minipage}[t]{8cm}$$\begin{array}{l}
\mbox{Camelia Petri\c sor -- Department of Mathematics,}\\
\mbox{Politehnica University of Timi\c soara},\\
 \mbox{P-ta Victoriei 2, 300 006, Timi\c soara, ROMANIA}\\
\mbox{E-mail: camelia.petrisor@upt.ro}\end{array}$$
\end{minipage}\\

\end{document}